\documentclass{article}

\usepackage{everyday_math}
\usepackage{wasysym}
\usepackage{cleveref}
\usepackage[T1]{fontenc}
\usepackage{times}
\usepackage{mathrsfs}
\usepackage{ltxtable}
\usepackage{tikz-cd}
\usepackage[super]{nth}

\usepackage{geometry}
\ifpdf
    \pdfcatalog { /PageMode (/UseOutlines)
                  /OpenAction (fitbh)  }
\fi

\title{$\alpha$ degrees as an automorphism base for the $\alpha$-enumeration degrees}
\ifpdf
  \author{\href{mailto:mmdt@leeds.ac.uk}{D\'avid Natingga}}

\else
  \author{D\'avid Natingga}
\fi

\newcommand{\cki}{\omega^{CK}_1}

\newcommand{\aeop}[1]{\Phi_{#1}}

\begin{document}

\maketitle

\setcounter{secnumdepth}{2}
\setcounter{tocdepth}{2}



\subsubsection{\Large Abstract}
Selman's Theorem \cite{selman1971arithmetical} in classical Computability Theory gives a characterization of the enumeration reducibility for arbitrary sets in terms of the enumeration reducibility on the total sets:
$A \le_e B \iff \forall X [X \equiv_{e} X \oplus \overline{X} \land B \le_{e} X \oplus \overline{X} \implies A \le_{e} X \oplus. \overline{X} ]$.
This theorem implies directly that the Turing degrees are an automorphism base of the enumeration degrees.
We lift the classical proof to the setting of the $\alpha$-Computability Theory to obtain the full generalization when $\alpha$ is a regular cardinal and partial results for a general admissible ordinal $\alpha$.

\section{$\alpha$-Computability Theory}
$\alpha$-Computability Theory is the study of the definability theory over G\"odel's $L_\alpha$ where $\alpha$ is an admissible ordinal. One can think of equivalent definitions on Turing machines with a transfinite tape and time \cite{koepke2005turing} \cite{koepke2007alpha} \cite{koepke2009ordinal} \cite{koepke_seyfferth2009ordinal} or on generalized register machines \cite{koepke2008register}. Recommended references for this section are \cite{sacks1990higher}, \cite{chong1984techniques}, \cite{maass1978contributions} and \cite{di1983basic}.

Classical Computability Theory is $\alpha$-Computability Theory where $\alpha = \omega$.

\subsection{G\"odel's Constructible Universe}

\begin{defn}\label{defn_goedels_constructible_universe}(G\"odel's Constructible Universe)\\
Define \emph{G\"odel's constructible universe} as $L := \bigcup_{\beta \in \mathrm{Ord}} L_\beta$ where $\gamma, \delta \in \mathrm{Ord}$, $\delta$ is a limit ordinal and:

$L_0 := \emptyset$,

$L_{\gamma + 1} := \mathrm{Def}(L_\gamma):=\{x | x \subseteq L_\gamma $ and $x$ is first-order definable over $L_\gamma\}$,

$L_\delta = \bigcup_{\gamma < \delta} L_\gamma$.
\end{defn}

\subsection{Admissibility}

\begin{defn}(Admissible ordinal\cite{chong1984techniques})\\
An ordinal $\alpha$ is \emph{$\Sigma_1$ admissible} (admissible for short) iff $\alpha$ is a limit ordinal and $L_\alpha$ satisfies $\Sigma_1$-collection:
$\forall \phi(x,y) \in \Sigma_1(L_\alpha). L_\alpha \models \forall u [\forall x \in u \exists y. \phi(x,y) \implies \exists z \forall x \in u \exists y \in z. \phi(x,y)]$ where $L_\alpha$ is the $\alpha$-th level of the G\"odel's Constructible Hierarchy (\cref{defn_goedels_constructible_universe}).
\end{defn}

\begin{eg}(Examples of admissible ordinals \cite{chong1984techniques} \cite{takeuti1965recursive})
\begin{itemize}
\item $\cki$ - Church-Kleene $\omega_1$, the first non-computable ordinal
\item every stable ordinal $\alpha$ (i.e. $L_\alpha \prec_{\Sigma_1} L$), e.g. $\delta^1_2$ - the least ordinal which is not an order type of a $\Delta^1_2$ subset of $\mathbb{N}$, \nth{1} stable ordinal
\item every infinite cardinal in a transitive model of $\mathrm{ZF}$
\end{itemize}
\end{eg}

\subsection{Basic concepts}

\begin{defn}\label{defn_ez_fin}
A set $K \subseteq \alpha$ is \emph{$\alpha$-finite} iff $K \in L_\alpha$.
\end{defn}

\begin{defn}($\alpha$-computability and computable enumerability)
\begin{itemize}
\item A function $f:\alpha \to \alpha$ is \emph{$\alpha$-computable} iff $f$ is $\Sigma_1(L_\alpha)$ definable.
\item A set $A \subseteq \alpha$ is \emph{$\alpha$-computably enumerable} ($\alpha$-c.e.) iff $A \in \Sigma_1(L_\alpha)$.
\item A set $A \subseteq \alpha$ is \emph{$\alpha$-computable} iff $A \in \Delta_1(L_\alpha)$ iff $A \in \Sigma_1(L_\alpha)$ and $\alpha-A \in \Sigma_1(L_\alpha)$.
\end{itemize}
\end{defn}

\begin{prop}\label{prop_bij_alpha_to_l_alpha}\cite{chong1984techniques}
There exists a $\Sigma_1(L_\alpha)$-definable bijection $b:\alpha \to L_\alpha$.
\qed
\end{prop}

Let $K_\gamma$ denote an $\alpha$-finite set $b(\gamma)$. The next proposition establishes that we can also index pairs and other finite vectors from $\alpha^n$ by an index in $\alpha$.

\begin{prop}\label{prop_bij_to_n_fold_product}\cite{maass1978contributions}
For every $n$, there is a $\Sigma_1$-definable bijection $p_n$:$\alpha \to \alpha \times \alpha \times ... \times \alpha$ (n-fold product).
\qed
\end{prop}

Similarly, we can index $\alpha$-c.e., $\alpha$-computable sets by an index in $\alpha$.
Let $W_e$ denote an $\alpha$-c.e. set with an index $e < \alpha$.

\begin{prop}\label{prop_alpha_finite_union}($\alpha$-finite union of $\alpha$-finite sets\footnote{From \cite{sacks1990higher} p162.})\\
$\alpha$-finite union of $\alpha$-finite sets is $\alpha$-finite, i.e. if $K \in L_\gamma$, then $\bigcup_{\gamma \in K} K_\gamma \in L_\alpha$.
\qed
\end{prop}

\subsection{Enumeration reducibility}
The generalization of the enumeration reducibility corresponds to two different notions - weak $\alpha$-enumeration reducibility and $\alpha$-enumeration reducibility.

\begin{defn}(Weak $\alpha$-enumeration reducibility)\\
$A$ is weakly $\alpha$-enumeration reducible to $B$ denoted as $A \le_{w\alpha e} B$ iff $\exists \Phi \in \Sigma_1(L_\alpha)$ st
$\Phi(B) = \{x < \alpha : \exists \delta < \alpha [ \langle x, \delta \rangle \in \Phi \land K_\delta \subseteq B]\}$.
The set $\Phi$ is called a weak $\alpha$-enumeration operator.
\end{defn}

\begin{defn}\label{defn_ae_reducibility}($\alpha$-enumeration reducibility)\\
$A$ is $\alpha$-enumeration reducible to $B$ denoted as $A \le_{\alpha e} B$ iff $\exists W \in \Sigma_1(L_\alpha)$ st
$\forall \gamma < \alpha [ K_\gamma \subseteq A \iff \exists \delta < \alpha [ \langle \gamma, \delta \rangle \in W \land K_\delta \subseteq B]]$.

Denote the fact that $A$ reduces to $B$ via $W$ as $A=W(B)$.
\end{defn}

\begin{fact}(Transitivity)\\
The $\alpha$-enumeration reducibility $\le_{\alpha e}$ is transitive.
But in general the weak $\alpha$-enumeration reducibility is not transitive.
\end{fact}

\subsection{Properties of $\alpha$-enumeration operator}
\begin{fact}\label{lemma_enum_from_aeop}
If $A \subseteq \alpha$, then $\aeop{e}(A) \le_{w\alpha e} A$.
\end{fact}

\begin{fact}\label{prop_enum_op_monotonicity}(Monotonicity)\\
$\forall e < \alpha \forall A, B \subseteq \alpha [A \subseteq B \implies \aeop{e}(A) \subseteq \aeop{e}(B)]$.
\qed
\end{fact}

\begin{prop}\label{prop_enum_op_witness}(Witness property)\\
If $x \in \aeop{e}(A)$, then $\exists K \subseteq A [K \in L_\alpha \land x \in \aeop{e}(K)]$.
\end{prop}
\begin{proof}
Note $\aeop{e}(A) := \bigcup \{K_\gamma : \exists \delta < \alpha [\langle \gamma, \delta \rangle \in W_e \land K_\delta \subseteq A\}$.
Thus if $x \in \aeop{e}(A)$, then $\exists \gamma < \alpha$ st $x \in K_\gamma$ and so $\exists \delta < \alpha [\langle \gamma, \delta \rangle \in W_e \land K_\delta \subseteq A]$.
Taking $K$ to be $K_\delta$ concludes the proof.
\end{proof}

\subsection{Totality}
\begin{defn}\footnote{From \cite{chong1984techniques} p8.}
The \emph{computable join} of sets $A, B \subseteq \alpha$ denoted $A \oplus B$ is defined to be

$A \oplus B := \{2a : a \in A\} \cup \{2b+1 : b \in B\}$.
\end{defn}

The computable join satisfies the usual properties of the case $\alpha=\omega$.

The generalization of the Turing reducibility corresponds to two different notions - weak $\alpha$ reducibility and $\alpha$ reducibility.

\begin{defn}(Total reducibilities)
\begin{itemize}
\item $A$ is $\alpha$-reducible to $B$ denoted as $A \le_{\alpha} B$ iff $A \oplus \overline{A} \le_{\alpha e} B \oplus \overline{B}$.
\item $A$ is weakly $\alpha$-reducible to $B$ denoted as $A \le_{w\alpha} B$ iff $A \oplus \overline{A} \le_{w\alpha e} B \oplus \overline{B}$.
\end{itemize}
\end{defn}

\begin{defn}(Total set)\\
A subset $A \subseteq \alpha$ is total iff $A \le_{\alpha e} \overline{A}$ iff $A \equiv_{\alpha e} A \oplus \overline{A}$.
\end{defn}

\subsection{Degree Theory}

\begin{defn}(Degrees)
\begin{itemize}
\item $\mathcal{D}_\alpha := \powerset{\alpha}/\equiv_\alpha$ is a set of \emph{$\alpha$-degrees}.
\item $\mathcal{D}_{\alpha e} := \powerset{\alpha}/\equiv_{\alpha e}$ is a set of \emph{$\alpha$-enumeration} degrees.
\end{itemize}
Induce $\le$ on $\mathcal{D}_\alpha$ and $\mathcal{D}_{\alpha e}$ by $\le_\alpha$ and $\le_{\alpha e}$ respectively.
\end{defn}

\begin{fact}(Embedding of the total degrees)\\
$\langle \mathcal{D}_\alpha, \le \rangle$ embeds into $\langle \mathcal{D}_{\alpha e}, \le \rangle$ via $\iota:\mathcal{D}_\alpha \hookrightarrow \mathcal{D}_{\alpha e}$, $A \mapsto A \oplus \overline{A}$.
\end{fact}

\begin{defn}(Total degrees)\\
Let $\iota:\mathcal{D}_\alpha \hookrightarrow \mathcal{D}_{\alpha e}$ be the embedding from above.
The total $\alpha$-enumeration degrees $\mathcal{TOT}_{\alpha e}$ are the image of $\iota$, i.e. $\mathcal{TOT}_{\alpha e} := \iota[\mathcal{D}_\alpha]$.
\end{defn}

\subsection{Megaregularity}
Megaregularity of a set $A$ measures the amount of the admissibility of a structure structure $\langle L_\alpha, A \rangle$, i.e. a structure extended by a predicate with an access to $A$.

\begin{note}(Formula with a positive/negative parameter)
\begin{itemize}
\item Let $B$ denote a set, $B^+$ its enumeration, $B^-$ the enumeration of its complement $\overline{B}$.
\item Denote by $\Sigma_1(L_\alpha, B)$ the class of $\Sigma_1$ formulas with a parameter $B$ or in $L_\alpha$.
\item A $\Sigma_1(L_\alpha, B)$ formula $\phi(\overline{x}, B)$ is $\Sigma_1(L_\alpha, B^+)$ iff $B$ occurs in $\phi(\overline{x}, B)$ only positively, i.e. there is no negation up the formula tree above the literal $x \in B$.
\item Similarly, a $\Sigma_1(L_\alpha, B)$ formula $\phi(\overline{x}, B)$ is $\Sigma_1(L_\alpha, B^-)$ iff $B$ occurs in $\phi(\overline{x}, B)$ only negatively.
\end{itemize}
\end{note}

\begin{defn}(Megaregularity)\\
Let $\mathcal{B} \in \{B, B^-, B^+\}$ and add $B$ as a predicate to the language for the structure $\langle L_\alpha, \mathcal{B} \rangle$.
\begin{itemize}
\item Then $\mathcal{B}$ is $\alpha$-\emph{megaregular} iff $\alpha$ is $\Sigma_1(L_\alpha, \mathcal{B})$ admissible iff the structure $\langle L_\alpha, \mathcal{B} \rangle$ is admissible,
i.e. every $\Sigma_1(L_\alpha, \mathcal{B})$ definable function satisfies the replacement axiom:
$\forall f \in \Sigma_1(L_\alpha, \mathcal{B}) \forall K \in L_\alpha. f[K] \in L_\alpha$.
\item $B$ is \emph{positively $\alpha$-megaregular} iff $B^+$ is $\alpha$-megaregular.
\item $B$ is \emph{negatively $\alpha$-megaregular} iff $B^-$ is $\alpha$-megaregular.
\end{itemize}
\end{defn}

If clear from the context, we just say \emph{megaregular} instead of $\alpha$-megaregular.

A person familiar with the notion of hyperregularity shall note that a set is
megaregular iff it is regular and hyperregular.

\begin{fact}(Megaregularity and definability)
\begin{itemize}
\item $B \in \Sigma_1(L_\alpha) \implies B^+$ is megaregular,
\item $B \in \Delta_1(L_\alpha) \implies B$ is megaregular.
\end{itemize}
\end{fact}

\begin{prop}\label{prop_alpha_fin_iff_bounded_and_alpha_comp_with_oracle}
Let $\mathcal{B} \in \{B, B^-, B^+\}$ be megaregular and let $A \subseteq \alpha$. Then:
$A \in L_\alpha$ iff $A \in \Delta_1(L_\alpha, \mathcal{B})$ and $A$ is bounded by some $\beta < \alpha$.
\end{prop}

\begin{proof}
$\implies$ direction is clear. For the other direction, assume that $A \in \Delta_1(L_\alpha, \mathcal{B})$ and $A \subseteq \beta < \alpha$ for some $\beta$.
WLOG let $A \not=\emptyset$ and let $a \in A$.
Define a function $f:\alpha \to \alpha$ by $f(x)=y :\iff x \not\in \beta \land y =a \lor x \in \beta \land [x \in A \land x=y \lor x \not\in A \land y=a]$.
Since  $A \in \Delta_1(L_\alpha, \mathcal{B})$, the function $f$ is $\Sigma_1(L_\alpha, \mathcal{B})$ definable.
By the megaregularity of $\mathcal{B}$, we have that $A=f[\beta] \in L_\alpha$ as required.
\end{proof}

\begin{cor}\label{hr_closure_and_deg_invariance}(Megaregularity closure and degree invariance)\\
i) If $A \le_{\alpha e} B$ and $B^+$ megaregular, then $A^+$ megaregular.\\
ii) If $A \equiv_{\alpha e} B$, then $[A^+$ megaregular iff $B^+$ megaregular $]$.\\
iii) If $A \le_{\alpha} B$ and $B$ megaregular, then $A$ megaregular.\\
iv) If $A \equiv_{\alpha} B$, then $[A$ megaregular iff $B$ megaregular $]$.
\qed
\end{cor}

\begin{prop}\label{prop_correspondence_ae_wae_sigma_1_def}(Correspondence between the $\alpha$-enumeration reducibilities)\\
We have the following implication diagram:\\
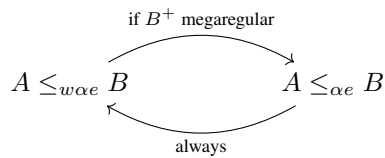
\begin{center}
\begin{tikzcd}
A \le_{w\alpha e} B \arrow[rr, bend left, "\text{if $B^+$ megaregular}"] & &
A \le_{\alpha e} B \arrow[ll, bend left, "\text{always}"]
\end{tikzcd}
\end{center}
\qed
\end{prop}

\section{Selman's theorem}\label{subsection_selman_theorem}
We generalize the theorem of Selman present in \cite{selman1971arithmetical}.

\begin{defn}(Odd enumeration and $\alpha$-finite part)
\begin{itemize}
\item Let $B \subseteq \alpha$.
The total function $f:\alpha \to \alpha$ is an \emph{odd enumeration} of $B$ iff $f[\{2\gamma + 1 < \alpha : \gamma < \alpha\}]=B$.
\item \emph{$B$ odd $\alpha$-finite part} is a function $\tau:[0,2s) \to \alpha$ for $s < \alpha$ st $\forall x < \alpha [2x+1 \in \mathrm{dom}(\tau) \implies \tau(2x+1) \in B]$.
\item Let $|\tau|$ denote the order type of $\mathrm{dom}(\tau)$, i.e. $|\tau|:=\mathrm{ot}(\mathrm{dom}(\tau))$.
\end{itemize}
\end{defn}

If $\mathrm{dom}(\tau)$ is an initial segment of $\alpha$, we have $|\tau|=\mathrm{dom}(\tau)$.
If $\tau:[0,2s) \to \alpha$ is a function, then $\mathrm{dom}(\tau)=2s$ and so $|\tau|=2s$.
If $\tau$ is a $B$ odd $\alpha$-finite part, then there is an odd enumeration $f:\alpha \to \alpha$ of $B$ st $\tau \subseteq f$.

\begin{lemma}\label{lemma_b_wae_from_odd_enum}
Let $f:\alpha \to \alpha$ be an odd enumeration of $B$.
Then $B \le_{w\alpha e} f$.
\end{lemma}
\begin{proof}
We have $B \le_{w\alpha e} f$ via $\aeop{} := \{\langle b, \delta \rangle : \exists \gamma < \alpha [2\gamma+1 < \alpha \land K_\delta = \{(2\gamma+1, b)\}]\}$.
\end{proof}

\begin{lemma}\label{lemma_f_inv_a_wae_f_selman}
Assume that $A \le_{w\alpha e} f$. Then $f^{-1}[A] \le_{w\alpha e} f$.
\end{lemma}
\begin{proof}
Note $f^{-1}[A]:=\{x < \alpha : f(x) \in A\}$.
Let $A \le_{w\alpha e} f$ via $\Psi \in \Sigma_1(L_\alpha)$.
Then $f^{-1}[A] \le_{w\alpha e} f$ via $\Phi := \{ \langle x, \delta \rangle : \exists y < \alpha [K_\delta = K_\epsilon \cup \{(x,y)\} \} \land \langle y, \epsilon \rangle \in \Psi]\} \in \Sigma_1(L_\alpha)$.
\end{proof}

\begin{defn}(Weak halting set)\\
The \emph{weak halting set} is defined as $K(A):=\{x < \alpha : x \in \Phi_x(A)\}$.
\end{defn}

\begin{prop}\label{prop_selman}\footnote{Generalized from \cite{soskova2010turingreducibilityandenumerationreducibility} for $\alpha=\omega$.}
Let $A, B \subseteq \alpha$ and $A \not\le_{w\alpha e} B$.
Assume that $A \oplus B \oplus K(\emptyset)$ is megaregular.
Then there exists an odd enumeration $f:\alpha \to \alpha$ of $B$ st $A \not\le_{w\alpha e} f$ and $B \le_{\alpha e}  f$.
\end{prop}

\begin{proof}

\subsubsection{Construction}
Note that $B \not= \emptyset$ and $B \not\equiv_{\alpha e} \emptyset$ as $A \not\le_{w\alpha e} B$.
We construct a sequence of $B$ odd $\alpha$-finite parts in $\alpha$ many stages st:
\begin{center}$\tau_0 \subseteq \tau_1 \subseteq ... \subseteq \tau_s \subseteq ...$\end{center}
In the end, the desired function $f:\alpha \to \alpha$ is defined as $f = \bigcup_{s < \alpha} \tau_s$.

Let $\tau_0 = \emptyset$. If $s$ is a limit ordinal, then let $\tau_s := \bigcup_{r < s} \tau_r$. Now assume that $\tau_s$ has been constructed, then at the stage $s$ construct $\tau_{s+1}$:
\begin{itemize}
\item Stage $s=2e$:\\
Set $\tau_{s+1} := \tau_s \cdot 0 \cdot b$ where $b = \mu y \{ y < \alpha: y \in B \land y \not\in \tau_s[\{2\gamma + 1 < |\tau_s| : \gamma < \alpha \}]\}$ and $0 \cdot b$ is the concatenation of $0$ and $b$. E.g. if $\tau = a \cdot b$, then $\tau(0)=a, \tau(1)=b$.
\item Stage $s=2e+1$:\\
Use $e$ and $\tau_s$ to define set $C$ as
\begin{center}$C := \{x < \alpha | \exists \rho \supseteq \tau_s [\rho$ is a $B$ odd $\alpha$-finite part and $x = \rho(|\tau_s|) \land |\tau_s| \in \aeop{e}(\rho)]\}$.\end{center}
As $C \le_{w\alpha e} B$, so $C \not= A$. Thus we have two cases:

\begin{itemize}
\item Case $\exists x < \alpha [x \in C \land x \not\in A]$: Then let $\tau_{s+1}$ be the minimal $\rho$ from $C$.
Note $\tau_{s+1}=\tau_s \cdot x \cdot b \cdot \sigma$ for some $b \in B$ where $\sigma$ is a $B$ odd $\alpha$-finite part.
\item Case $\exists x < \alpha [x \not\in C \land x \in A]$: Then let $\tau_{s+1} := \tau_s \cdot x \cdot b$ for some $b \in B$.
\end{itemize}
\end{itemize}

\subsubsection{Verification}
By the two cases above we have for all $e, s, x < \alpha$:
\begin{instat}\label{stat_x_xor_membership_selman_prop}
s=2e+1 \land x = \tau_{s+1}(|\tau_s|) \implies x \in C \land x \not\in A \lor x \not\in C \land x \in A
\end{instat}

Note that $f = \bigcup_{s < \alpha} \tau_s$ is an odd enumeration of $B$.
This is ensured by stages $s=2e < \alpha$.
If $b \in B$, then a pair $(2\gamma+1, b)$ is added to $f$ for some $2\gamma + 1 < \alpha$ at the stage $s=2b < \alpha$ at latest.

Moreover, $A \not\le_{w\alpha e} f$. For suppose not, then $A \le_{w\alpha e} f$.
Hence $f^{-1}[A] \le_{w \alpha e} f$ by \cref{lemma_f_inv_a_wae_f_selman} and so there is $e < \alpha$ st $f^{-1}[A]=\aeop{e}(f)$ and thus:
\begin{instat}\label{stat_eq_selman_prop}\forall l < \alpha [f(l) \in A \iff l \in f^{-1}[A] \iff l \in \aeop{e}(f)]\end{instat}

Consider the stage $s = 2e+1$.
Let $l = |\tau_s|$. Note $\tau_s \subseteq f$.
\begin{itemize}
\item Case 1: $l \in f^{-1}[A] \implies f(l) \in A \implies l \in \aeop{e}(f)$ using \cref{stat_eq_selman_prop}.
By the witness property of an $\alpha$-enumeration operator there exists $B$ odd $\alpha$-finite part $\rho$ of $f$ st $\rho \supseteq \tau_s \land l \in \aeop{e}(\rho) \land \rho(l)=f(l)$. So $f(l) \in C$.
Hence $f(l) \in C \cap A$.
Also by \cref{stat_x_xor_membership_selman_prop} we have $f(l) \in C \land f(l) \not\in A \lor x \not\in C \land f(l) \in A$.
This is a contradiction.
\item Case 2: $l \not\in f^{-1}[A] \implies f(l) \not\in A \implies l \not\in \aeop{e}(f)$ using \cref{stat_eq_selman_prop}.
By the monotonicity of an $\alpha$-enumeration operator there is no $B$ odd $\alpha$-finite part $\rho$ of $f$ st 
$\rho \supseteq \tau_s  \land l \in \aeop{e}(\rho)$.
So $f(l) \not\in C$.
Hence $f(l) \not\in A \land f(l) \not\in C$.
By \cref{stat_x_xor_membership_selman_prop} we have $f(l) \in C \land f(l) \not\in A \lor x \not\in C \land f(l) \in A$.
This is a contradiction.
\end{itemize}
Hence in any case $A \not\le_{w\alpha e} f$ as required.

Let $g : \alpha \to \alpha$ be defined by $g : s \mapsto \gamma$ where $K_\gamma = \tau_s$.
During the construction we use the oracle $A \oplus B \oplus K(\emptyset)$, hence $g \in \Sigma_1(L_\alpha, A \oplus B \oplus K(\emptyset))$.
We show that $g$ is well-defined and that $\tau_s$ is $\alpha$-finite at a limit stage $s$.
Let $s < \alpha$ be a limit stage. Then $g[s]=I \in L_\alpha$ since $s \in L_\alpha$, $g_{|s} \in \Sigma_1(L_\alpha, A \oplus B \oplus K(\emptyset))$ and by the megaregularity of the oracle $A \oplus B \oplus K(\emptyset))$.
Hence $\tau_s = \bigcup_{r < s} \tau_r = \bigcup_{\gamma \in I} K_\gamma$ is $\alpha$-finite as required.
Therefore $\forall s < \alpha. \tau_s \in L_\alpha$ as needed.

Note that $f \in \Delta_1(L_\alpha, A \oplus B \oplus K(\emptyset))$ since the construction of $f$ uses the oracle $A \oplus B \oplus K(\emptyset)$.
By that and the megaregularity of $A \oplus B \oplus K(\emptyset)$, also $f$ must by megaregular.
By \cref{lemma_b_wae_from_odd_enum} we have $B \le_{w\alpha e} f$.
By the megaregularity of $f$, we have $B \le_{\alpha e} f$ as required.
\end{proof}

\begin{thm}\label{thm_selman}(Selman's theorem for admissible ordinals)\\
Let $A, B \subseteq \alpha$ and let $A \oplus B \oplus K(U)$ be megaregular. Then:\\
$A \le_{\alpha e} B \iff \forall X [X \equiv_{\alpha e} X \oplus \overline{X} \land B \le_{\alpha e} X \oplus \overline{X} \implies A \le_{\alpha e} X \oplus \overline{X} ]$
\end{thm}

Part of the following proof is adapted from the classical case present in \cite{soskova2010turingreducibilityandenumerationreducibility}.

\begin{proof}
$\Rightarrow$ direction is by the transitivity of $\le_{\alpha e}$.

For the $\Leftarrow$ direction assume that $\forall X [X \equiv_{\alpha e} X \oplus \overline{X} \land B \le_{\alpha e} X \oplus \overline{X} \implies A \le_{\alpha e} X \oplus \overline{X}]$.
We want to show that $A \le_{\alpha e} B$.
Assume not, then $A \not\le_{\alpha e} B$ and so $A \not\le_{w \alpha e} B$ as $A \oplus B \oplus K(U)$ is megaregular.
Then by \cref{prop_selman} and the megaregularity of $A \oplus B \oplus K(U)$ there exists a total function $f$ st $A \not\le_{w\alpha e} f$, but $B \le_{\alpha e} f$.
But $f$ is total and so $A \le_{\alpha e} f$ which is a contradiction to the statement $A \not\le_{w\alpha e} f$.
Hence $A \le_{\alpha e} B$ as required.
\end{proof}

\begin{cor}\label{thm_selman_for_alpha_regular_card}(Selman's theorem for regular cardinals)\\
Let $\alpha$ be a regular cardinal. Then for any $A, B \subseteq \alpha$ we have:\\
$A \le_{\alpha e} B \iff \forall X [X \equiv_{\alpha e} X \oplus \overline{X} \land B \le_{\alpha e} X \oplus \overline{X} \implies A \le_{\alpha e} X \oplus \overline{X} ]$
\end{cor}
\begin{proof}
If $\alpha$ is a regular cardinal, then every subset of $\alpha$ is megaregular.
Hence $A \oplus B \oplus K(U)$ is megaregular.
The remaining proof of the corollary follows from \cref{thm_selman}.
\end{proof}

\section{$\mathcal{TOT}_{\alpha e}$ as an automorphism base for $\mathcal{D}_{\alpha e}$}\label{section_aut_base}

\begin{defn}(Automorphism base)\\
TFAE:
\begin{itemize}
\item The subset $B \subseteq \mathrm{dom}(\mathcal{M})$ is an automorphism base of the model $\mathcal{M}$.
\item $\forall f, g \in \mathrm{Aut}(\mathcal{M}) [f_{|B} = g_{|B} \implies f = g]$
\item $\forall f \in \mathrm{Aut}(\mathcal{M}) [f_{|B} = 1_{|B} \implies f = 1]$
\end{itemize}
\end{defn}

\begin{thm}\label{thm_automorphism_base}
Assume $\alpha$ is a regular cardinal.
The total degrees $\mathcal{TOT}_{\alpha e}$ are an automorphism base for $\mathcal{D}_{\alpha e}$.
\end{thm}

\begin{proof}
0. Assume $\alpha$ is a regular cardinal.

1. $\forall a, b \in \mathcal{D}_{\alpha e}[ a \le b \iff \forall x \in \mathcal{TOT}_{\alpha e} (b \le x \implies a \le x)]$ by $0$ and \cref{thm_selman_for_alpha_regular_card}.

2. $\forall a,b \in \mathcal{D}_{\alpha e} [a = b \iff \forall x \in \mathcal{TOT}_{\alpha e} (b \le x \iff a \le x)]$ by 1.

3. Assume $f \in \mathrm{Aut}(\mathcal{D}_{\alpha e})$.

4. $\forall a, b \in \mathcal{D}_{\alpha e} [ a \le b \iff f(a) \le f(b)]$ by 3.

5. Assume $\forall x \in \mathcal{TOT}_{\alpha e}. f(x) = x$.

6. Assume $y \in \mathcal{D}_{\alpha e}$.

7. $\forall x \in \mathcal{TOT}_{\alpha e}. (f(y) \le f(x) \iff y \le x)$ by 4.

8. $\forall x \in \mathcal{TOT}_{\alpha e}. (f(y) \le x \iff y \le x)$ by 5, 7.

9. $f(y)=y$ by 2, 8.

10. $\forall y \in \mathcal{D}_{\alpha e}. f(y) = y$ by 6, 9.

11. $\mathcal{TOT}_{\alpha e}$ is an automorphism base for $\mathcal{D}_{\alpha e}$ by 3, 5, 10.
\end{proof}

\section{Acknowledgements}
The author would like to thank Mariya Soskova for the explanation of the proof in the classical case, i.e. $\alpha = \omega$.

The author was supported by Hausdorff Research Institute for Mathematics during Hausdorff Trimester Program \emph{Types, Sets and Constructions}.

\bibliographystyle{plain}
\bibliography{References/references} 
\addcontentsline{toc}{chapter}{References} 

\end{document}